\documentclass[12pt,oneside]{amsart}
\usepackage[utf8]{inputenc}
\usepackage{amsmath,amstext,amsopn,amssymb, amsfonts, amsthm, mathtools}
\usepackage[dvipsnames]{xcolor}
\usepackage{geometry}
\usepackage{bbm}
\usepackage{pdfsync}
\usepackage{tikz}
\usepackage{enumitem}

\usepackage[pdftex,bookmarks,colorlinks,breaklinks]{hyperref}  
\definecolor{dullmagenta}{rgb}{0.4,0,0.4}   
\definecolor{darkblue}{rgb}{0,0,0.4}
\definecolor{darkgreen}{rgb}{0,0.4,0}
\hypersetup{linkcolor=darkblue,citecolor=blue,filecolor=dullmagenta,urlcolor=darkblue} 

\newtheorem{TheoremLetter}{Theorem}
{}

\newtheorem{definition}{Definition}
\newtheorem*{definition*}{Definition}
\newtheorem{theorem}{Theorem}
\newtheorem*{theorem*}{Theorem}
\newtheorem{lemma}[theorem]{Lemma}
\newtheorem*{lemma*}{Lemma}
\newtheorem{proposition}[theorem]{Proposition}
\newtheorem{corollary}[theorem]{Corollary}

\newtheorem{remark}[theorem]{Remark}
\newtheorem*{remark*}{Remark}

\newtheorem*{problem*}{Problem}
\numberwithin{equation}{section}
\numberwithin{theorem}{section}

\makeatletter
\newcommand{\customlabel}[2]{%
   \protected@write \@auxout {}{\string \newlabel {#1}{{#2}{\thepage}{#2}{#1}{}} }%
   \hypertarget{#1}{#2}
}
\makeatother

\def\XXint#1#2#3{{\setbox0=\hbox{$#1{#2#3}{\int}$}
     \vcenter{\hbox{$#2#3$}}\kern-.5\wd0}}

\newcommand{\dist}{\operatorname{dist}_{\mathbb{H}}}
\newcommand{\Dil}{\operatorname{Dil}}

\makeindex
\begin{document}

\title[Sparse domination and the strong maximal function]{Sparse domination and the strong\\maximal function}

\author{Alex Barron}
\address{Alexander Barron \\
Department of Mathematics, Brown University, 151 Kassar House, Providence, RI, USA.}
\email{abarron@math.brown.edu}

\author[Conde-Alonso]{Jos\'e M. Conde-Alonso}
\address{Jos\'e M. Conde-Alonso\\
Department of Mathematics, Brown University, 151 Kassar House, Providence, RI, USA.
}
\email{jconde@math.brown.edu}

\author{Yumeng Ou}
\address{Yumeng Ou\\
Department of Mathematics, Box B6-230, One Bernard Baruch Way, New York, NY, USA}
\email{yumeng.ou@baruch.cuny.edu}

\author{Guillermo Rey}
\address{Guillermo Rey\\
127 Vincent Hall, 206 Church St. SE Minneapolis, MN, USA}
\email{reyg@umn.edu}

\begin{abstract}
We study the problem of dominating the dyadic strong maximal function by
  $(1,1)$-type sparse forms based on rectangles with sides parallel to the
  axes, and show that such domination is impossible. Our proof relies on an
  explicit construction of a pair of maximally separated point sets with
  respect to an appropriately defined notion of distance.
\end{abstract}

\allowdisplaybreaks

\maketitle


\section{Introduction} \label{section:introduction}

Recent years have seen a great deal of work around the concept of
\textit{sparse domination}. Perhaps the easiest domination result is the one
for the dyadic maximal function
$$
  M_{\mathcal{D}} f(x) = \sup_{Q \ni x} \frac{1}{|Q|}\int_Q |f|,
$$
where the supremum is taken over the family $\mathcal{D}$ of all dyadic cubes in $\mathbb{R}^d$. One can use the Calder\'on-Zygmund decomposition at varying heights to obtain
(for each non-negative locally integrable function $f$) a collection $\mathcal{S}$ of dyadic cubes from the same filtration $\mathcal{D}$ such that
$$
  M_{\mathcal{D}} f \leq C \sum_{Q \in \mathcal{S}} \langle f \rangle_Q \mathbbm{1}_Q,
$$
where $\langle f \rangle_Q = \frac{1}{|Q|}\int_Q f$, and $\mathcal{S}$ is \textit{sparse}:
\begin{definition} \label{definition:Sparse}
  A collection of sets $\mathcal{S}$ is $\eta$-sparse if for all $Q \in \mathcal{S}$ there exists
  a subset $E(Q) \subseteq Q$ such that $|E(Q)| \geq \eta |Q|$ and
  the collection $\{E(Q)\}$ is pairwise disjoint.\footnote{In what follows, we will usually omit the parameter $\eta$ and just say ``sparse" instead of ``$\eta$-sparse".}
\end{definition}
It is worth mentioning that this notion was already used by Sawyer \cite{Sawyer1982} and by Christ-Fefferman \cite{ChristFefferman1983} to prove certain weighted estimates, although the language that they used was different. Moreover, the result readily extends to more general martingale filtrations and a similar estimate allows one to sparsely dominate martingale transforms as well (see \cite{Lacey2017}).

The recent interest in sparse domination results was sparked by Lerner (see for example
\cite{Lerner2010,Lerner2013a,Lerner2013}) in connection with the famous $A_2$ conjecture
as it gave a particularly simple proof of the --now-- theorem
(first settled by Hyt\"onen in \cite{HytonenA2} after numerous efforts by the mathematical community).
Lerner's original result was a \textit{dual} form of sparse domination:
\begin{equation} \label{eq:Lerner}
  \|Tf\|_{\mathbb{X}} \lesssim \sup_{\mathcal{S}} \Bigl\| \sum_{Q \in \mathcal{S}} \langle |f| \rangle_Q \mathbbm{1}_Q \Bigr\|_{\mathbb{X}},
\end{equation}
where the supremum is taken over all sparse collections (with respect to a finite number of
dyadic grids), and $\mathbb{X}$ is any \textit{Banach function space}. 

The number of sparse domination results is by now very large and it would be impossible to give a
detailed survey of all of them in a reasonable amount of text. But let us note that \eqref{eq:Lerner}
has been generalized in many ways, and the typical result is that for every pair of (sufficiently nice) functions
$f$ and $g$ there exists a sparse collection of cubes $\mathcal{S}$ such that
\begin{equation} \label{equation:DualDomination}
  |\langle T f, g \rangle| \lesssim \sum_{Q \in \mathcal{S}}
    \langle |f| \rangle_Q \langle |g| \rangle_{Q} |Q|
\end{equation} or more generally 

\begin{equation} \label{equation:DualDomination2}
|\langle T f, g \rangle| \lesssim \sum_{Q \in \mathcal{S}}
\langle |f| \rangle_{Q,r} \langle |g| \rangle_{Q,s} |Q|,
\end{equation} where $\langle |f| \rangle_{Q,r} = \langle |f|^{r} \rangle_{Q}^{\frac{1}{r}}$ and $ 1 \leq r,s < \infty$. The $(r,s)$ sparse forms of type \eqref{equation:DualDomination2} arise when studying operators that fall outside the scope of classical Calder\'{o}n-Zygmund theory, and which may not satisfy the full range of strong and weak-type $L^{p}$ estimates implied by \eqref{equation:DualDomination}. The operators that have been studied include but are not limited to rough singular integral operators, Bochner-Riesz multipliers, spherical maximal functions, singular integrals along manifolds, pseudodifferential operators, and the bilinear Hilbert transform. 
See for example
\cite{BeltranCladek2017arxiv,
CondeRey2016,
Lacey2017,
Lacey2017arxiv,
BeneaBernicot2017,
LaceyMena2017,
CCPO2017,
CPO2018a,
CPO2018b}. Sparse domination provides a fine quantification of the mapping properties of operators that carries much more information than $L^p$ bounds. In particular, it is a very effective way to obtain sharp quantitative weighted estimates and has also been used to study previously unknown endpoint behaviors (see for instance \cite{KrauseLacey}).  

In the present work we study the problem of sparsely dominating the bi-parameter analogue of the the dyadic maximal function. This operator is the dyadic strong maximal function:

\begin{equation} \label{definition:strongMaximalFunction}
  \mathcal{M}_Sf(x) := \sup_{R \ni x} \langle |f| \rangle_R,
\end{equation}
where the supremum is taken over all dyadic rectangles containing $x$ with sides parallel to the axes. The strong maximal function is one of the most important operators in the theory of multi-parameter singular integrals, associated with which is an underlying non-isotropic dilation structure. This class of operators arise naturally in the theory of summation of multiple Fourier series, several complex variables and certain boundary value problems, and is a first step into the study of operators with more complicated dilation structure (e.g. Zygmund dilations). However, much less is known about these operators compared to the one-parameter case, for example the sharp weighted bound, thus it is natural to ask whether the sparse domination technique can be introduced into the multi-parameter setting to help the study. Unfortunately, this seems to be very difficulty due to the fact that one of the key ingredients in standard proofs of sparse domination, the stopping time argument, is missing in the multi-parameter setting.

We first note that in the study of multi-parameter sparse domination, the natural geometric objects to consider are axes-parallel rectangles instead of cubes, since $\mathcal{M}_S$ can be large on axis-parallel rectangles of arbitrary eccentricity. Then, one observes that the strong maximal function is indeed dominated by sparse forms (based on rectangles) when restricted to a single point mass. Indeed, one can make sense of \eqref{definition:strongMaximalFunction} when applied to finite positive measures, and then it is easy to see that
\begin{equation} \label{MofDelta}
   \mathcal{M}_S(\delta_0)(x,y) \leq \frac{1}{|x||y|}.
\end{equation}
And this function can actually be dominated by a sparse operator: if
we define $I_m = [0,2^m)$ and $J_m = [2^{m-1}, 2^m)$ then
\begin{align*}
  \frac{1}{|x||y|} &= \sum_{m,n \in \mathbb{Z}} \frac{1}{|x||y|} \mathbbm{1}_{J_m \times J_n}(x,y) \\
    &\leq \sum_{m,n \in \mathbb{Z}} 2^{2-m-n} \mathbbm{1}_{J_m \times J_n}(x,y) \\
    &= 4\sum_{m,n \in \mathbb{Z}} \langle \delta_0 \rangle_{I_m \times I_n} \mathbbm{1}_{I_m \times I_n}(x,y).
\end{align*}
Note that the collection $\mathcal{S} = \{I_m \times I_n:\, m,n \in \mathbb{Z}\}$ is sparse since
for every $R = I_m \times I_n$ we can define $E(R) = J_m \times J_n$ and this satisfies the conditions
of Definition \ref{definition:Sparse} with $\eta=1/4$. 

This example is relevant because (approximate) point masses are extremal examples for the weak-type behavior of $\mathcal{M}_S$ and the Hardy-Littlewood maximal functions near the $L^{1}$ endpoint. Moreover, we know from the one-parameter theory that there is a connection between sparse bounds and weak-type endpoint estimates; for example, a $(1,1)$ sparse bound of type \eqref{equation:DualDomination} implies that $T$ maps $L^{1}$ into weak $L^{1}$ (see the appendix in \cite{CCPO2017}). Taking the above discussion into account, it is natural to ask whether or not $\mathcal{M}_{S}$ admits a $(1,1)$ sparse bound of the type \begin{equation} \label{equation:sparseFormM} |\langle \mathcal{M}_S f, g \rangle| \lesssim \sum_{R \in \mathcal{S}} \langle |f| \rangle_R \langle |g| \rangle_R |R|, \end{equation} where the collection $\mathcal{S}$ consists of dyadic axis-parallel rectangles. We also note that there is no immediate contradiction implied by a bound of type \eqref{equation:sparseFormM}; indeed, an estimate of the type $$ |\langle T f, g \rangle| \lesssim \sum_{R \in \mathcal{S}} \langle |f| \rangle_R \langle |g| \rangle_R |R|$$ implies that $T$ maps $L \log L$ into weak $L^{1}$, but does \textit{not} imply that $T$ maps $L^{1}$ into weak $L^{1}$ (the proof is similar to the argument in the appendix of \cite{CCPO2017}).

Our main result answers the question raised above in the negative: there can be no domination by positive sparse forms of
the type \eqref{equation:DualDomination} for the strong maximal function.

\begin{TheoremLetter} \label{MainTheorem}
  For every $C>0$ and $0<\eta<1$ there exist a pair of compactly supported integrable functions $f$ and $g$ such
  that
  $$
    |\langle \mathcal{M}_Sf, g \rangle| \geq C \sum_{R \in \mathcal{S}} \langle |f| \rangle_R \langle |g| \rangle_R |R|,
  $$
  for all $\eta$-sparse collections $\mathcal{S}$ of dyadic rectangles with sides parallel to the axes.
\end{TheoremLetter}


The proof of theorem \ref{MainTheorem} is based on the construction of pairs of
extremal functions for which the sparse bound cannot hold. These extremal
examples take advantage of the behavior of $\mathcal{M}_S$ when applied to sums
of several point masses. Indeed, when we apply $\mathcal{M}_S$ to a point mass we get level
sets that look like dyadic \emph{stairs}, as displayed in Figure
\ref{fig:stairs}. These stairs are sparse, in fact the measure of their union
is proportional to the number of rectangles that form them. We shall describe a
way to place many point masses sufficiently far from each other in such a way
that these stairs are all necessary to dominate $\mathcal{M}_S$, but are packed
too tightly to be sparse.  

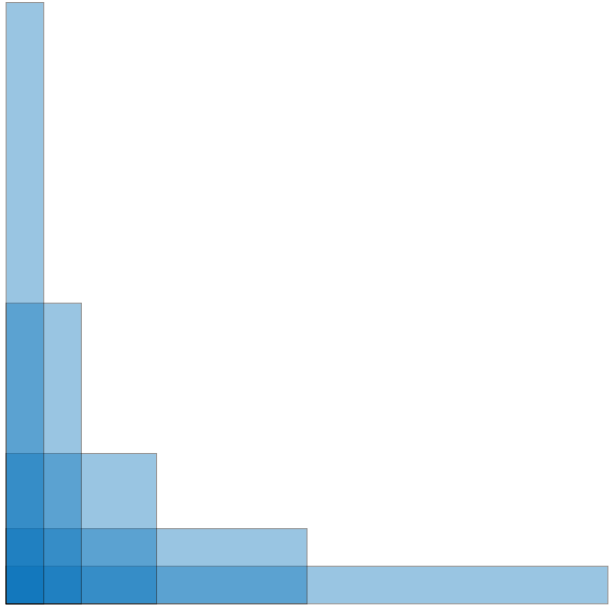
\begin{figure}
\begin{tikzpicture}[scale=0.5]
  \filldraw[NavyBlue, draw=black, opacity=0.4] (0,0) rectangle (16,1);
  \filldraw[NavyBlue, draw=black, opacity=0.4] (0,0) rectangle (8,2);
  \filldraw[NavyBlue, draw=black, opacity=0.4] (0,0) rectangle (4,4);
  \filldraw[NavyBlue, draw=black, opacity=0.4] (0,0) rectangle (2,8);
  \filldraw[NavyBlue, draw=black, opacity=0.4] (0,0) rectangle (1,16);

\end{tikzpicture}
  \caption{Dyadic \emph{stairs}}
  \label{fig:stairs}
\end{figure}

In particular, we shall construct a set of points $\mathcal{P}$ which are
maximally separated with respect to a quantity that captures the biparametric
structure of the problem. Our first extremal function is then the sum of point
masses at the points in $\mathcal{P}$. As we said above, this will guarantee
that $\mathcal{M}_S$ applied to it is uniformly large on the unit square. We
complement the construction of the set $\mathcal{P}$ with another set,
$\mathcal{Z}$, with the following structure: for each point $p \in \mathcal{P}$
there exist a large amount of points $z \in \mathcal{Z}$ that are both near $p$
and far away from all the other points in $\mathcal{P}$. Moreover, the family
of points in $\mathcal{Z}$ associated with a given $p \in \mathcal{P}$ are not
clustered together, but are instead spread out over the boundary of a hyperbolic
ball centered at $p$. Our second extremal function is then the normalized sum of
point masses at the points in $\mathcal{Z}$. We will show that the placement of
the points in $\mathcal{Z}$ implies that if a form given by a family
$\mathcal{S}$ dominates the pair $\langle \mathcal{M}_S f, g \rangle$ then the rectangles in $\mathcal{S}$ need to cover the
portion of the stair centered at $p$ that contains each $z$, for all $z \in
\mathcal{Z}$. But then the resulting collection of rectangles cannot be sparse.


We remark that our result actually extends to the situation in which we take
slightly larger averaged norms of $f$. In particular our methods allow us to push 
Theorem \ref{MainTheorem} to the case of Orlicz $\phi(L)$-norms with
$\phi(x) = x \log(x)^\alpha$ for $\alpha < 1/2$. We carry out this extension in Theorem
\ref{thm:sec5norm}.

We have learned that our construction of $\mathcal{P}$ is closely connected to the fields of discrepancy, combinatorics, and computational learning. In particular it is a special construction of a \emph{low-discrepancy} sequence similar to
those of Hammersley, see for example Definition 3.44 of \cite{DickPillichshammer2010}. These low-discrepancy sequences are also closely connected to the notion of $\varepsilon$-nets, though our point of view is
in a sense opposite to the one usually taken in the theory of $\varepsilon$-nets: we are interested
in \emph{lower} bounds on the cardinality of the intersection of these sets with rectangles.
The notion of $(t,m,s)$-nets is closer to our purposes, see \cite{DickPillichshammer2010}. Our results show that these kind of sets are particularly well-suited to the study of the strong maximal function.
However, the theory of discrepancy and $\varepsilon$-nets is much more general
so it would be very interesting to find further connections between other types of maximal operators and the fields above. In particular Zygmund-type maximal operators (defined like the strong maximal function, but with rectangles with fewer
``parameters'' than the ambient dimension, see
\cite{Cordoba1980,Soria1986,FeffermanPipher2005,Stokolos2008}) could be
amenable to this connection.


The structure of the rest of the paper is the following. In Section
\ref{section:LowerBounds}, we prove lower bounds for $\mathcal{M}_S$ and, in
the opposite direction, we prove an upper bound for sparse forms when acting on
the same extremal functions. We assume the existence of the key sets of points
$\mathcal{P}$ and $\mathcal{Z}$, which are needed to construct the
aforementioned extremal functions. The sets $\mathcal{P}$ and $\mathcal{Z}$ are constructed in detail in Section
\ref{section:ConstructionOfSets}, completing the proof of Theorem
\ref{MainTheorem}. Finally, in Section \ref{section:FurtherRemarks} we show how
our results extend to higher dimensions by a tensor product argument. We also show that the strong maximal function is in a sense
supercritical for $L^1$-sparse forms, which allows us to slightly strengthen
our result.

{\bf Acknowledgements.}  The authors would like to thank Jill Pipher and Dmitriy Bilyk for useful discussions regarding the content of this paper. The third named author is supported by NSF-DMS \#1764454. 


\section{Lower bounds for $\mathcal{M}_S$ and upper bounds for sparse forms} \label{section:LowerBounds}

We start by introducing a quantity which is intimately related with the geometry of the bi-parameter setting: for every pair of points $p,q$ in $[0,1)^2$ we set 
$$
\dist(p,q) = \inf \{|R|^{1/2}:\, R \in \mathcal{D} \times \mathcal{D} \text{ is a dyadic rectangle containing } p \text{ and } q\}.
$$
One can also define a distance between two (dyadic) rectangles:
\[
\dist(R_1,R_2):=\inf_{p\in R_1, q\in R_2}\dist(p,q).
\]
Note that $\dist$ is not a distance function, and not even a quasidistance, as the triangle inequality is completely false in general. We will, however, refer to it as \textit{the distance} between two points. Using this language, the strong maximal function of a point mass \eqref{MofDelta} can now be written as
$$
  \mathcal{M}_S(\delta_0)(p) = \frac{1}{\dist(p,0)^2}.
$$
We are going to study the precise behavior of $\mathcal{M}_S$ when applied to uniform probability measures
concentrated on certain point sets: given a finite set of points $\mathcal{F}$ let
\begin{equation}\label{equation:defmu}
  \mu = \mu_{\mathcal{F}} = \frac{1}{\#(\mathcal{F})} \sum_{p \in \mathcal{F}} \delta_p.
\end{equation}

In general evaluating the strong maximal function of a sum of point masses is a computationally
difficult problem: for every point $p$ one has to consider all possible subsets of $\mathcal{F}$
whose minimal enclosing rectangle contains $p$. However, if we know a priori that $\mathcal{F}$ is
$\varepsilon$-separated (with respect to $\dist$), and that it contains sufficiently many points,
then this forces a certain uniform distribution of $\mu$ at large scales,
reducing the problem to treating only small rectangles. In particular, let $\varepsilon > 0$, suppose that $\mathcal{F}=\mathcal{P}$ is such that $\#(\mathcal{P}) \geq C \varepsilon^{-2}$, and that it is $\varepsilon$-separated:
$$
  \dist(p,q) \geq \varepsilon \quad \forall p \neq q \text{ in } \mathcal{P}.
$$
Then for any rectangle with $|R| \geq \varepsilon^2$ we have
\begin{equation*}
  \langle \mu \rangle_R = \frac{\mu(R)}{|R|} \leq \frac{\#(R \cap \mathcal{P})}{C|R|\varepsilon^{-2}} \lesssim \frac{|R|\varepsilon^{-2}}{C |R| \varepsilon^{-2}} = \frac{1}{C}.
\end{equation*}

If instead $|R| < \varepsilon^2$ then $R$ can contain at most one point from $\mathcal{P}$ and so
for any $p$ for which $\dist(p,\mathcal{P}) \lesssim 1$ we have

\begin{equation}\label{equation:estimateForM}
  M_S(\mu)(p) \sim \frac{1}{\#(\mathcal{P})} \sup_{q \in \mathcal{P}} \frac{1}{\dist(p,q)^2}.
\end{equation}

As we mentioned in the introduction, if one can find another set of points $\mathcal{Z}$ each of which is very close to exactly one point $p$ from $\mathcal{P}$ |and which are not very close to one another, so that they do not form large clusters|, then we can force a contradiction that makes sparse domination fail. The rest of this section is devoted to giving lower bounds for $\mathcal{M}_S$ assuming the existence of such sets of points $\mathcal{P}$ and $\mathcal{Z}$. In Section \ref{section:UpperBounds} we exploit the intuition given above and use it to give upper bounds for sparse forms that, together with the lower bounds from this section, will ultimately yield a contradiction.

The following two theorems give the full description of $\mathcal{P}$ and $\mathcal{Z}$. We postpone their proofs to Section \ref{section:ConstructionOfSets}.

\begin{theorem} \label{theorem:PPoints}
  For every $m \geq 0$ there exists a collection of points $\mathcal{P} \subset [0,1)^2$ such that
  \begin{equation} \label{eq:PPoints:SeparationOfPoints} \tag{P.1}
    \dist(p,q) \geq 2^{-m}, \quad \forall p \neq q \text{ in } \mathcal{P},
  \end{equation}
  and
  \begin{equation} \label{eq:PPoints:NumberOfPoints} \tag{P.2}
    \#(\mathcal{P}) = 2^{2m+1}.
  \end{equation}
\end{theorem}

\begin{theorem} \label{theorem:ZPoints}
  For every $k \ll m$ and every $2^{-m}$-separated set $\mathcal{P} \subset [0,1)^2$
  there exists a set $\mathcal{Z} \subset [0,1)^2$ satisfying the following properties:
  \begin{enumerate}[label=\textup{(Z.\arabic*)}]
    \item $\#(\mathcal{Z}) \gtrsim m 2^{2m}$. \label{theorem:ZPoints:Size}
    \item \label{theorem:ZPoints:DistToP} For every $z \in \mathcal{Z}$ there
      exists exactly one point $p(z) \in \mathcal{P}$ such that
      $$
        \dist(p(z),z)^2 < 2^{-2m-1}.
      $$
    \item \label{theorem:ZPoints:DistToPz} For every $z \in \mathcal{Z}$, 
      $$
        \dist(p(z),z)^2 \sim 2^{-2m-k}.
      $$
    \item \label{theorem:ZPoints:Separation} For every dyadic rectangle $R$ with $|R|
      \leq 2^{-2m - 2}$ that intersects $\mathcal{P}$ we have
      $$
        \#(R \cap \mathcal{Z}) \lesssim k.
      $$
   \item \label{theorem:ZPoints:UniformDist} For every dyadic rectangle $R$ with $|R| \geq 2^{-2m-1}$ we have
      $$
        \#(R \cap \mathcal{Z}) \lesssim 2^{2m}mk |R|.
      $$
  \end{enumerate}
  The implied constants are independent of $k$ and $m$.
\end{theorem}

\begin{remark} 
	A simple pigeonholing argument shows that if $\mathcal{P}$ is the collection from Theorem \ref{theorem:PPoints} and $R$ is a dyadic rectangle with $|R| \geq 2^{-2m-1}$, then $\#(R \cap \mathcal{P}) = 2^{2m+1}|R|$.  
\end{remark}

\begin{remark} 
Above, we deliberately omitted the precise dependence of $m$ on $k$. It turns out that any $m \geq k +C$ is admissible as one can check from the proofs in Section \ref{section:ConstructionOfSets}. This justifies our choices of $m$ and $k$ in the next subsections.
\end{remark}

\subsection{Lower bound for $\mathcal{M}_S$}

For the rest of this section $k \ll m$ will be some fixed large numbers,
and $(\mathcal{P}, \mathcal{Z})$ will be the sets given by Theorems \ref{theorem:PPoints} and
\ref{theorem:ZPoints}. Also, $\mu=\mu_{\mathcal{P}}$ and $\nu=\nu_{\mathcal{Z}}$ will always denote the associated uniform probability measures introduced in \eqref{equation:defmu}. The proofs of Theorems \ref{theorem:ZPoints} and \ref{theorem:ZPoints} in the next section show that $\mu$ and $\nu$ can be arbitrarily well-approximated by $L^{1}$ functions, in particular because we will be able to choose the points from $\mathcal{P}$ and $\mathcal{Z}$ in small cubes. Because of this we will prove Theorem \ref{MainTheorem} with $f$ and $g$ replaced by the measures $\mu$ and $\nu$, respectively, and the full result can then be recovered by a simple limiting argument. 

\begin{proposition} \label{theorem:LowerBoundForM}
  Under these conditions we have
  \begin{equation}
    \langle \mathcal{M}_S(\mu), \nu \rangle \gtrsim 2^k.
  \end{equation}
\end{proposition}
\begin{proof}
  The positivity of $\mathcal{M}_S$ makes the proof almost trivial since we do not need to care about
  possible interactions among the points in $\mathcal{P}$. In particular
  for all $z \in \mathcal{Z}$ let $p(z)$ be the point guaranteed by \ref{theorem:ZPoints:DistToP} of Theorem \ref{theorem:ZPoints}.
  Then, by \eqref{equation:estimateForM} and \eqref{theorem:ZPoints:DistToPz},
  $$
    \mathcal{M}_S(\mu)(z) \geq \frac{1}{\#(\mathcal{P})\dist(p(z),z)^2} \sim \frac{2^{2m+k}}{\#(\mathcal{P})}.
  $$
  Now by \eqref{eq:PPoints:NumberOfPoints} of Theorem \ref{theorem:PPoints},
  $$
    M_S(\mu)(z) \gtrsim 2^k
  $$
  and the claim follows from the fact that $\nu$ is a uniform probability measure over the set $\mathcal{Z}$.
\end{proof}


\subsection{Upper bound for sparse forms} \label{section:UpperBounds}

We now prove upper bounds for sparse forms when acting on $\mu$ and $\nu$. Our results here are written in the language of Carleson sequences:

\begin{definition*}
  Let $\alpha$ be a non-negative function defined on all rectangles that is zero for all
  but a finite collection of rectangles. We say that $\alpha$ is $\Lambda$-Carleson if for all
  open sets $\Omega$ we have
  $$
    \sum_{R \subseteq \Omega} \alpha_R |R| \leq \Lambda |\Omega|,
  $$
  where the sum is taken over all rectangles contained in $\Omega$.
\end{definition*}
We call a collection of dyadic rectangles $\mathcal{S}$ a $\Lambda$-Carleson collection if the sequence
$$
  \alpha_R =
    \begin{cases}
      1 \quad \text{ if } R \in \mathcal{S}, \\
      0 \quad \text{ otherwise}
    \end{cases}
$$
is a $\Lambda$-Carleson sequence. The notions of sparse and Carleson collections are equivalent: a $\Lambda$-Carleson collection is $\Lambda^{-1}$-sparse and
vice-versa, as was shown in \cite{Hanninen2017arxiv} (see also \cite{Dor1975}).

\begin{proposition} \label{theorem:UpperBound}
  For every $\Lambda$-Carleson collection $\mathcal{S}$ and all $k$ and $m$ we have
  \begin{equation}
    \sum_{R \in \mathcal{S}} \langle \mu \rangle_R \langle \nu \rangle_R|R|
      \lesssim \Lambda k\Bigl( 1 + \frac{2^{2k}}{m} \Bigr).
  \end{equation}
\end{proposition}

\begin{proof}
For any rectangle $R$ let $\overline{R}$ be the intersection of $R$ with the unit square $[0,1)^2$.
We will show that
\begin{equation} \label{eq:ProductOfAverages}
  \langle \mu \rangle_R \langle \nu \rangle_R \lesssim k \Bigl(1 + \frac{2^{2k}}{m}\Bigr)
  \Bigl( \frac{|\overline{R}|}{|R|}\Bigr)^2.
\end{equation}

Assume first that $R$ is \emph{large}, i.e.: $|\overline{R}| \geq 2^{-2m-1}$.
Then by the $2^{-m}$-separation of $\mathcal{P}$
$$
  \langle \mu \rangle_R = \frac{1}{2^{2m+1}} \frac{\#(\mathcal{P} \cap \overline{R})}{|R|}  = \frac{|\overline{R}|}{|R|}.
$$

Similarly, for $\nu$ we have by \ref{theorem:ZPoints:UniformDist}:
$$
  \langle \nu \rangle_R \lesssim k\frac{|\overline{R}|}{|R|}.
$$

Therefore, for large rectangles we have
$$
  \langle \mu \rangle_R \langle \nu \rangle_R \lesssim k\Bigl( \frac{|\overline{R}|}{|R|}\Bigr)^2.
$$

If instead $R$ is \emph{small}, i.e.: $|\overline{R}| < 2^{-2m-1}$, then we have to be more careful.
If
$$
  \langle \mu \rangle_R \langle \nu \rangle_R \neq 0
$$
then $\overline{R}$ must contain at least one point $p \in \mathcal{P}$.
In fact, since $\overline{R}$ is sufficiently small, there exists exactly one
$p \in \mathcal{P} \cap \overline{R}$.
Hence
$$
  \langle \mu \rangle_{\overline{R}} \lesssim \frac{1}{2^{2m}|\overline{R}|}.
$$
Since $\langle \nu \rangle_R \neq 0$ the rectangle $\overline{R}$ must contain
at least one point from $\mathcal{Z}$.
This implies a lower bound on the size of $\overline{R}$. To see this observe that if $z$ is any point
in $R \cap \mathcal{Z}$, then by the smallness of $R$ and \ref{theorem:ZPoints:DistToP} we must have
$p = p(z)$, and hence $\dist(p,z) \sim 2^{-2m-k}$ by \ref{theorem:ZPoints:DistToPz}, therefore
$$
  |\overline{R}| \gtrsim 2^{-2m-k}.
$$

Now, by \ref{theorem:ZPoints:Separation} we have
$$
  \#(\overline{R} \cap \mathcal{Z}) \lesssim k,
$$
so
$$
  \langle \mu \rangle_{\overline{R}} \lesssim 2^k
  \quad \text{and} \quad
  \langle \nu \rangle_{\overline{R}} \lesssim \frac{k2^k}{m},
$$
from which inequality \eqref{eq:ProductOfAverages} follows.

Now we can finish the proof by splitting $\mathcal{S}$ into families $\mathcal{S}_j$ as follows:
$$
  \mathcal{S}_j = \{R \in \mathcal{S}: \, 2^{-j-1}|R| \leq |\overline{R}| < 2^{-j} |R|\}.
$$
Note that the rectangles in $\mathcal{S}_j$ are contained in
$$
  \Omega_j = \{p \in \mathbb{R}^2:\, M_S(\mathbbm{1}_{[0,1)^2}) \gtrsim 2^{-j}\},
$$
which, by the weak-type boundedness of the strong maximal function, satisfies
\begin{equation} \label{eq:sizeOfOmegaj}
|\Omega_j| \lesssim j2^j.
\end{equation}

Then by estimate \eqref{eq:ProductOfAverages}, the Carleson condition, and \eqref{eq:sizeOfOmegaj},
\begin{align*}
  \sum_{R \in \mathcal{S}} \langle \mu \rangle_R \langle \nu \rangle_R |R| 
    &=  \sum_{j=0}^\infty \sum_{R \in \mathcal{S}_j} \langle \mu \rangle_R \langle \nu \rangle_R |R| \\
    &\lesssim k\Bigl( 1 + \frac{2^{2k}}{m} \Bigr) \sum_{j=0}^\infty 2^{-2j} \sum_{R \in \mathcal{S}_j} |R| \\
    &\leq k\Bigl( 1 + \frac{2^{2k}}{m} \Bigr) \sum_{j=0}^\infty 2^{-2j} \Lambda |\Omega_j| \\
    &\lesssim \Lambda k\Bigl( 1 + \frac{2^{2k}}{m} \Bigr) \sum_{j=0}^\infty j 2^{-j} \\
    &\lesssim \Lambda k\Bigl( 1 + \frac{2^{2k}}{m} \Bigr).
\end{align*}
\end{proof}

\begin{remark}
An examination of the proof above shows that we do not use the full power of the Carleson condition. In particular, the collection $\mathcal{S}$ can be assumed to satisfy the $\Lambda$-Carleson packing condition with respect to the unit  
cube and the level sets $\Omega_j$, but it does not need to be $\Lambda$-Carleson with respect to any other open sets (and in particular need not be $\Lambda$-Carleson at small scales). 
\end{remark}

We can now formally conclude the proof of Theorem \ref{MainTheorem}
conditionally on Theorems \ref{theorem:PPoints} and \ref{theorem:ZPoints}.

\begin{proof}[Proof of Theorem \ref{MainTheorem}]
  Choose $m = k2^{2k}$. Proposition \ref{theorem:LowerBoundForM} implies that
  $$
    \langle M_S(\mu), \nu\rangle \gtrsim 2^k,
  $$
  so Proposition \ref{theorem:UpperBound} forces $\Lambda \to \infty$ if we make $k \to \infty$, which leads to a contradiction. 
\end{proof}


\section{Construction of the extremal sets} \label{section:ConstructionOfSets}

In this section, we will sometimes need to work with the projections of a rectangle $R$ onto the axes. To that end, if $R=I \times J$ we denote
$$
\pi_1(R):= I, \;\; \pi_2(R):=J.
$$

\noindent For dyadic intervals $I$ we let $\widehat{I}$ denote the dyadic parent of $I$. We also let $I^{(j)}$ denote the $j$-fold dyadic dilation of $I$, so that $I^{(j)}$ is the dyadic interval containing $I$ with $|I| = 2^{j}|I|$. 

\subsection{Construction of $\mathcal{P}$ and the proof of Theorem \ref{theorem:PPoints}} The following observation will be useful in the construction:

\begin{lemma}\label{lem:dist}
Let $R_1=I_1\times J_1, R_2=I_2\times J_2$ be two dyadic rectangles such that $I_1\cap I_2=J_1\cap J_2=\emptyset$. Then
\[
 \dist(p_1,p_2)=\dist(q_1,q_2),\quad \forall p_1,q_1\in R_1,\,p_2,q_2\in R_2.
\]
\end{lemma}
\begin{proof}
  Given any $p_1\in R_1$, $p_2\in R_2$, it suffices to show that any dyadic rectangle $R=I\times J$ containing both $p_1$ and $p_2$ needs to contain $R_1$ and $R_2$. Since $I\cap I_i\neq\emptyset$, $i=1,2$, and $I_1\cap
  I_2=\emptyset$, one has $I\supset I_i$, $i=1,2$. The same holds for $J$: $J\supset J_i$ for $i=1,2$.
\end{proof}

\begin{proof}[Proof of theorem \ref{theorem:PPoints}] We are going to show the following claim by induction: for every non-negative integer $m$ there exist $2^{2m+1}$ dyadic squares $Q_1^m, \dots, Q_{2^{2m+1}}^m$ in $[0,1)^2$ satisfying
  \begin{enumerate}
    \item $\ell(Q_i^m) = 2^{-2m-1}$ for all $i$,
    \item $\dist(Q_i^m, Q_j^m) \geq 2^{-m}$ for all $i \neq j$ \label{distance}.
  \end{enumerate}
 Assuming the claim, it is enough to take
 $$
 \mathcal{P}= \left\{p_Q: Q \in \left\{Q_1^m, \dots, Q_{2^{2m+1}}^m\right\}\right\},
 $$
where $p_Q$ denotes the center of the cube $Q$ |in fact, any point of $Q$ would work|. Then, by Lemma \ref{lem:dist} we immediately get \eqref{eq:PPoints:SeparationOfPoints} and \eqref{eq:PPoints:NumberOfPoints} and we end the proof.

We turn to the proof of the claim. The case $m=0$ is easy: it suffices to take $Q_1^0=[0,1/2)^2$ and $Q_2^0=[1/2,1)^2$. Assume now by induction that the theorem is true for $m-1$. Scale the family obtained in that step by $1/2$ in both dimensions and place a translated copy of the result in each square $Q \in \mathcal{D}_1([0,1)^2)$. Denote the cubes so constructed by $P_1^Q, \dots, P_{2^{2m-1}}^Q$. By the dilation invariance of $\dist$ we have
  \[
    \ell(P_i^Q) = 2^{-1} \cdot 2^{-2(m-1)-1} = 2^{-2m}.
  \]
Therefore, it is enough to choose a first-generation child from each of these squares $\{P_i^Q\}_{i,Q}$ in such a way that \eqref{distance} is satisfied. Write $\mathcal{D}_1([0,1)^2)=\{Q_0^0, Q_0^1, Q_1^0, Q_1^1\}$, where the children are listed in the order of upper left, upper right, lower left, and lower right. We first consider $Q_0^0$ and $Q_1^1$. For each $P_i^Q$ with $Q \in \{Q_0^0, Q_1^1\}$ we choose an arbitrary first-generation child and call it $Q_i^Q$. By induction we have
  \[
    \dist(Q_i^Q, Q_j^Q) \geq \dist(P_i^Q, P_j^Q)\geq 2^{-1} \cdot 2^{-(m-1)}
      = 2^{-m} \quad \text{for all } i \neq j,
  \]
  while the distance between any square in $Q_0^0$ and any other square in $Q_1^1$ must be exactly $1$ according to Lemma \ref{lem:dist}.

  Now we must choose the children from the squares in the other diagonal. This
  choice is more delicate since the squares from the first diagonal could be
  much closer. By the pigeonhole principle, for each $P = P^Q_i$, $Q\in \{Q_0^1,
  Q_1^0\}$, there exist exactly two other squares, $P_j^{Q_0^0}$ and
  $P_k^{Q_1^1}$, which are at distance at most $2^{-m}$. Indeed, according to
  Lemma \ref{lem:dist}, the distance condition forces $P_j^{Q_0^0}, P_k^{Q_1^1}$
  to be in the same row (or column) as $P_i^Q$, while each row (or column) of
  $Q_0^0$ (or $Q_1^1$) contains exactly one square by induction hypothesis. In
  fact, the distance between $P_i^Q$ and $P_j^{Q_0^0}$ is precisely equal to
  $0$, and the same holds for $P_k^{Q_1^1}$.

  We now project $Q_j^{Q_0^0}$ and $Q_k^{Q_1^1}$, the children chosen from
  $P_j^{Q^0_0}$ and $P_k^{Q^1_1}$ onto $P$. Note that their projection leaves
  exactly one first-generation child of $P$ untouched, which we select as
  $Q_i^Q$. It suffices to show that the distance from $Q_i^Q$ to any square in
  $Q_0^0$ is larger than $2^{-m}$, since the case of $Q_1^1$ is symmetric and any
  other combination can be dealt with in the same way as above. To see this, if
  the square in $Q_0^0$ does not come from the $P_j^{Q_0^0}$ chosen above, the
  distance has to be larger than $2^{-m}$ by Lemma \ref{lem:dist}. Otherwise,
  the square is exactly $Q_j^{Q_0^0}$. By the construction above, both its horizontal projection and its vertical projection and the respective ones of $Q_i^Q$ are disjoint, which allows one to apply Lemma
  \ref{lem:dist} once more to conclude the proof.
\end{proof}

\subsection{Construction of $\mathcal{Z}$} \label{subsection:buildingZ} In the construction, there are two aspects that will require
special care: each point in $\mathcal{Z}$ has to be close to exactly one point from
$\mathcal{P}$ (but not too close), while two different points in $\mathcal{Z}$ cannot be too close to each other.
We begin by constructing certain special rectangles inside which we shall place the points forming
$\mathcal{Z}$.

\begin{definition*}
  We say that a rectangle $R$ is a standard rectangle if it belongs to the collection
  \[
    \mathcal{E} = \{R \subseteq [0,1)^2: R \text{ dyadic, } |R| = 2^{-2m-2}, \text{ and }
    R \cap \mathcal{P} \neq \emptyset\}.
  \]
  If $R\in \mathcal{E}$ we let $p_R$ denote the unique point in 
  $R \cap \mathcal{P}$. Also let $\mathcal{E}_{p}$ denote the collection of
  standard rectangles containing $p\in \mathcal{P}$.
\end{definition*}

One can obtain $\mathcal{E}_p$ by area-preserving dilations of one fixed $R \in \mathcal{E}_p$. Indeed, for any dyadic $R = I \times J$ and any $j \in \mathbb{Z}$ define
\[
  \Dil_j R = I' \times J',
\]
where $I' \times J'$ is the unique dyadic rectangle in $\mathcal{E}_p$ of dimensions
$2^{j}|I| \times 2^{-j}|J|$. Then it is easy to see that for each standard rectangle $R$, $\mathcal{E}_p = \{\Dil_j R \subseteq [0,1)^2\}$. In particular, for any $p \in \mathcal{P}$ there are approximately $m$ distinct standard
rectangles in $\mathcal{E}_p$. Given $x,y \in \mathbb{R}$, let 
$$
\delta(x,y) = \inf \{|Q|:\, Q \in \mathcal{D} \text{ is a dyadic interval containing } x \text{ and } y\}.
$$ 
Also, for dyadic intervals $K$ define
$$
\delta(x,K):=\inf_{y\in K}\delta(x,y).
$$
Now suppose $R = I \times J \in \mathcal{E}$, with $p_R= p = (x,y)$. Given $k
\geq 1$, we define $I_{(k)}$ to be the largest dyadic interval $I'$ such that
$$
\delta(x, I') = 2^{-k+1}|I|,
$$
and define $J_{(k)}$ similarly. Note that $I_{(1)}$ is the child of $I$
that does not contain $x$, and for all $k\geq 1$, $I_{(k)}$ is the unique dyadic
offspring of $I$ of generation $k$ satisfying $x\notin I_{(k)}$ and $x\in \widehat{I_{(k)}}$ (the dyadic parent of $I_{(k)}$).

The following result is the main technical lemma needed for our construction.

\begin{lemma} \label{lemma:pointFinder2} Fix an integer $k \geq 1$ with $k \ll m$. 
  For any dyadic rectangle $R = I\times J \in \mathcal{E}_p$
  there exists a sub-rectangle $R^* \subset I_{(1)} \times J_{(k)}$ such that
  for all $q \neq p$ in $\mathcal{P}$:
  $$
    \dist(q, R^*)^2 \geq 2^{-2m-1}.
  $$
\end{lemma}

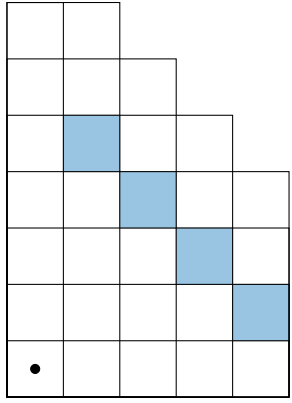
\begin{figure}
\begin{tikzpicture}[scale=0.75]
  \filldraw (0.5,0.5) circle [radius=0.08];

  \draw (0,0) rectangle (5,1);
  \draw (0,0) rectangle (5,2);
  \draw (0,0) rectangle (5,3);

  \draw (0,0) rectangle (5,4);
  \draw (0,0) rectangle (4,5);
  \draw (0,0) rectangle (3,6);
  \draw (0,0) rectangle (2,7);
  \draw (0,0) rectangle (1,7);

  \filldraw[NavyBlue, draw=black, opacity=0.4] (1,4) rectangle (2,5);
  \filldraw[NavyBlue, draw=black, opacity=0.4] (2,3) rectangle (3,4);
  \filldraw[NavyBlue, draw=black, opacity=0.4] (3,2) rectangle (4,3);
  \filldraw[NavyBlue, draw=black, opacity=0.4] (4,1) rectangle (5,2);
\end{tikzpicture}
  \caption{log-log sketch of some standard rectangles. In blue are the subrectangles $I_{(1)} \times J_{(k)}$
  \label{fig:StdRects}
  with $k=3$.}
\end{figure}

We will apply this lemma to construct the points $\mathcal{Z}$ used in the proof of the main theorem as follows: given $R=I\times J \in \mathcal{E}$, we choose one point $z_R$ in $R^{\ast} \subset I_{(1)} \times J_{(k)}$ and let $\mathcal{Z} = \bigcup_{R \in \mathcal{E}}z_{R}$.  Figure \ref{fig:StdRects} gives some intuition about the situation. Note that this immediately guarantees properties \ref{theorem:ZPoints:Size}, \ref{theorem:ZPoints:DistToP} and \ref{theorem:ZPoints:DistToPz} of Theorem \ref{theorem:ZPoints}. This placing also allows us to prove the remaining two properties. We start by proving directly \ref{theorem:ZPoints:Separation}:

\begin{lemma}\label{lemma:Zcard} For every $R \in \mathcal{E}$
  we have
  $$
    \#(R \cap \mathcal{Z}) \lesssim k.
  $$ 
\end{lemma} 
\begin{proof}
  First observe that if $z_{R'} \in R$ with $R, R' \in \mathcal{E}$ then by
  Lemma \ref{lemma:pointFinder2} we must have $p_{R'} = p_{R} = p$. So let
  $\mathcal{Z}_R$ be defined by
  \[
    \mathcal{Z}_R = \{ T \in \mathcal{E}_{p}:\, z_T \in R\}.
  \]
  It suffices to show that $\#(\mathcal{Z}_R) \lesssim k$.
  
  Observe that any $T \in \mathcal{Z}_R$ must have $\pi_1(T) \subseteq
  \pi_1(R)$. Indeed, if $\pi_1(T)\supsetneq \pi_1(R)$, by definition $z_T\in
  \pi_1(T)_{(1)}\times \pi_2(T)_{(k)}$ and one has $\pi_1(R)\subset
  \pi_1(T)_{(1)}$. But this is impossible since
  $\pi_1(T)_{(1)}\subseteq \pi_1(T)\setminus \pi_1(R)$.
  In addition, there must hold 
  $$
  |\pi_2(T)| \leq 2^{k-1} |\pi_2(R)|,
  $$ 
  as the smallest interval containing both $\pi_2(p_T)$ and $\pi_2(z_T)$ is $\widehat{\pi_2(T)_{(k)}}$. The desired estimate then follows from the fact that there are $O(k)$ such standard rectangles.
\end{proof}

\subsection{Proof of lemma \ref{lemma:pointFinder2} and \ref{theorem:ZPoints:UniformDist}} \label{subsection:pointFinder2}

In this subsection, given two dyadic rectangles $R_1,R_2$ that are not contained in one another but which do intersect at some nontrivial set, we say that $R_2$ intersects $R_1$ \textit{horizontally} if 
  $$
   \pi_2(R_2) \subsetneq \pi_2(R_1),
  $$
  and similarly $R_2$ intersects $R_1$ \textit{vertically} if 
  $$
    \pi_1(R_2) \subsetneq \pi_1(R_1).
  $$
One can see that these two options are mutually exclusive for $R_1$ and $R_2$. Also, one can see that if $R_2$ intersects $R_1$ horizontally then $\pi_1(R_1) \subsetneq \pi_1(R_2)$, and if $R_2$ intersects $R_1$ vertically then $\pi_2(R_1) \subsetneq \pi_2(R_2)$. We start by recording the following information that will be used later:

\begin{lemma} \label{lemma:LocalizePs}
  Let $R = I \times J \in \mathcal{E}$ be such that $I \times J^{(\ell+1)}
  \subseteq [0,1)^2$, $\ell\geq 1$. Then
  \[
    I_{(\ell)} \times J^{(\ell+1)}
  \]
  contains exactly one $q \in \mathcal{P}$, $q\neq p_R$, and furthermore
  \[
    q \in I_{(\ell)} \times (J^{(\ell+1)} \setminus J^{(\ell)}).
  \]
\end{lemma}
\begin{proof}
We have that $|\widehat{I_{(\ell)}} \times J^{(\ell+1)}| = 2^{-\ell+1}2^{\ell+1}|I||J| = 2^{-2m},$ so there are exactly two points from $\mathcal{P}$ in $\widehat{I_{(\ell)}} \times J^{(\ell+1)}$. One of these points must be $p_{R}$, so the other one is $q$. To prove the second part of the assertion, note that we cannot have $q \in (\widehat{I_{(\ell)}} \backslash I_{(\ell)}) \times J^{(\ell+1)}$ or $q \in I_{(\ell)} \times J^{(\ell)}$, since in each case we would have $p_R$ and $q$ contained in a dyadic rectangle of area smaller or equal than $2^{-2m-1}$, violating \eqref{eq:PPoints:SeparationOfPoints}.
\end{proof}

\begin{lemma}
  \label{lemma:findMaximalRectangles}
  Let $R = I \times J \in \mathcal{E}$. For any $\ell\geq 1$ with
  $2^{\ell+1}|J|\leq 1$, there exists exactly one $T = W \times
  H \in \mathcal{E}$ intersecting $R$ vertically such that $p_{T} \neq p_{R}$, $W\subset
  I_{(\ell)}$, and $|W| = 2^{-1}|I_{(\ell)}|$. 
\end{lemma}

\begin{proof}
  By Lemma \ref{lemma:LocalizePs} the rectangle $I_{(\ell)} \times
  J^{(\ell+1)}$ contains exactly one $q \in \mathcal{P}$ and $q \neq p_{R}$. Moreover, we know that $q \in I_{(\ell)} \times (J^{(\ell+1)} \backslash J^{(\ell)})$. We define $W$ as the child of $I_{(\ell)}$ that contains $\pi_1(q)$, and choose $H$ so that $\pi_2(q) \in H$ and $T = W \times H$ has area $2^{-2m-2}$. In particular, we
  have $H = J^{(\ell+1)}$. By construction, $T \in \mathcal{E}$. 

  We now claim that $T$ is the unique standard rectangle intersecting $R$ vertically with $p_{T} \neq p_{R}$, $W\subset I_{(\ell)},$ and $|W| = \frac{1}{2}|I_{(\ell)}|$. By contradiction, suppose there existed another $T' = W' \times H'$ satisfying the same properties. Then since $W, W' \subset I_{(\ell)}$ with $|W| = |W'|$ we know that $T$ and $T'$ are disjoint and therefore $p_T \neq p_{T'}$. Since $T'$ is a standard rectangle and $T,T'$ both intersect $R$ vertically, we also know that $H' = H = J^{(\ell+1)}$. But this contradicts Lemma \ref{lemma:LocalizePs}: that there can only be one $q \in \mathcal{P}$ in $I_{(\ell)} \times J^{(\ell+1)}$ with $q \neq p_{R}$, so $p_T = p_{T'}$, which is a contradiction. 
\end{proof}


Lemma \ref{lemma:findMaximalRectangles} holds for rectangles intersecting $R$ horizontally with exactly the same proof, which we omit for brevity. Assume now that the rectangle $R=I \times J$ has been fixed. Let
$$
\mathcal{E}_R^{\ell,v} = \left\{ T \in \mathcal{E}: T \mbox{ intersects } R \mbox{ vertically and } \pi_1(T)\subset I_{(\ell)} \right\}.
$$
and
$$
\mathcal{E}_R^{\ell,h} = \left\{ T \in \mathcal{E}: T \mbox{ intersects } R \mbox{ horizontally and } \pi_2(T)\subset J_{(\ell)} \right\}.
$$
\noindent We will analyze the vertical rectangles $\mathcal{E}_R^{\ell,v}$; essentially the same arguments apply to the horizontal collection $\mathcal{E}_R^{\ell,h}$.

If $S$ and $T$ are in $\mathcal{E}_R^{\ell,v}$, we say that $S\preceq T$ if $\pi_1(S) \subset \pi_1 (T)$ (an analogous order can be defined in $\mathcal{E}_R^{\ell,v}$). Denote by $\mathcal{E}_R^{\ell,v,*}=\{T_1, T_2, \dots\}$ the set of maximal elements with respect to the ordering $\preceq$. Assume they are ordered so that the sequence $a_i:=|\pi_1(T_i)|$ is non-increasing.

\begin{lemma}
  \label{lemma.exponentialDecay}
  The sequence $a_i$ satisfies
  $$
  a_i = 2^{-i}|I_{(\ell)}|, 1 \leq i \leq m- \ell.
  $$
\end{lemma}

\begin{proof}
  By Lemma \ref{lemma:findMaximalRectangles} we find that $T_{1}$ satisfies $\pi_1(T_1) \subset I_{(\ell)}$ and $a_1 = 2^{-1}|I_{(\ell)}|$. 
  Moreover, we know that there can be no other $T_{j}$ with $a_j = a_1$, and so in particular $a_2 \leq 2^{-2}|I_{(\ell)}|$. 

  We assume inductively that the desired result holds for $T_1, \ldots, T_{i-1}$,
  and aim to show that $a_i = 2^{-i}|I_{(\ell)}|$. Note that $\{\pi_1(T_1),\ldots,
  \pi_1(T_{i-1})\}$ are pairwise disjoint. Then by induction 
  $$
  \pi_1(T_i) \subset I_{(\ell)} \backslash \bigsqcup_{j=1}^{i-1} \pi_1(T_j),
  $$
  and we also have that $a_j=2^{-1}a_{j-1}$ and $a_1 = 2^{-1}|I_{(\ell)}|$. Therefore, $\pi_1(T_i)$ must be contained in some interval $K$ with $|K|= 2^{-i+1}|I_{(\ell)}|$. Moreover, by induction $K$ and $\pi_1(T_{i-1})$ must have the same dyadic parent. 

We first show that $a_i<|K|$. To see this, suppose by contradiction that $\pi_1(T_i)=K$. Then $a_i=a_{i-1}$ and 
  $$
  \widehat{\pi_1(T_i)}=\widehat{\pi_1(T_{i-1})}.
  $$
Since both $T_{i-1}$ and $T_i$ must also have the same height, their union is a dyadic rectangle of area $2^{-2m-1}$ which contains both $p_{T_{i-1}}$ and $p_{T_i}$, which is impossible by \eqref{eq:PPoints:SeparationOfPoints}.

Finally, we have to show that $a_i \geq 2^{-1} a_{i-1}$ or, equivalently, that there exists $T \in \mathcal{E}_R^{\ell,v,*}$ with $ |\pi_1(T)|=2^{-1} a_{i-1}$. To that end, let $\tilde{T}_{i-1} \in \mathcal{E}_{p_{T_{i-1}}}$ be the unique rectangle with $\pi_1(\tilde{T}_{i-1}) =\widehat{K}$. If $2^{2}|\pi_2(\tilde{T}_{i-1})|\leq 1$, one can apply Lemma \ref{lemma:findMaximalRectangles} with $\ell=1$ to $\tilde{T}_{i-1}$ to conclude that there is exactly one $T \in \mathcal{E}$ intersecting $\tilde{T}_{i-1}$ vertically with $p_T \neq p_{T_{i-1}}$ and 
$$
  |\pi_1(T)| = 2^{-1}|\widehat{K}_{(1)}| = 2^{-i}|I_{(\ell)}|.
$$ 
$T_{i} = T$ by maximality, and this closes the induction. If on the other hand $2^{2}|\pi_2(\tilde{T}_{i-1})|>1$, then we claim that $i > m-\ell$ and we have finished. Indeed, if $T_i$ exists in this case, one has 
$$
|\pi_1(T_i)|<|K|=2^{-1}|\pi_1(\tilde{T}_{i-1})|.
$$ 
Therefore, 
$$|\pi_2(T_i)|>2|\pi_2(\tilde{T}_{i-1})|\geq 2^2|\pi_2(\tilde{T}_{i-1})|>1,
$$
which means it is impossible for $T_i$ to be contained in $[0,1)^2$.
\end{proof}

\begin{proof}[Proof of lemma \ref{lemma:pointFinder2}] It is enough to show the
  following slightly stronger claim: for any pair of integers $k,\ell \geq 1$,
  any $R=I\times J \in \mathcal{E}$ contains a non-empty rectangle $E\times F$ such that
$$
    E \times F \subseteq  I_{(\ell)} \times J_{(k)} \backslash
    \big(\bigcup_{\substack{ R' \in \mathcal{E} \\ p_{R'} \neq p_{R} }} R' \big).
$$
Any $R' \in \mathcal{E}$ which intersects $I_{(\ell)} \times J_{(k)}$ must be in one of the two collections $\mathcal{E}_R^{\ell,v}$, $\mathcal{E}_R^{k,h}$, when $p_R \neq p_{R'}$.

We start by constructing $E$, so now we only need to take rectangles belonging to $\mathcal{E}_R^{\ell,v,*}$ into account. From Lemma \ref{lemma.exponentialDecay} we know that $a_i=|\pi_1(T_i)|$ satisfies $a_i = 2^{-i}|I_{(\ell)}|$. It thus follows that $|\pi_2(T_i)|$ must increase in size exponentially, and therefore the total number of $T_i \in \mathcal{E}_R^{\ell,v,*}$ is bounded by some constant that depends on $R$ (since all rectangles are contained in $[0,1)^2$). Since there are only finitely many $\pi_1(T_i) \subset I_{(\ell)}$, the estimate on their sizes from Lemma \ref{lemma.exponentialDecay} implies that there must be some dyadic subinterval $E \subset I_{(\ell)}$ with 
  $$
  \pi_1(T_i) \cap E = \emptyset \mbox{ for all }i,
  $$
as desired. The same procedure yields an interval $F \subset J_{(k)}$ which no rectangles $T'\in \mathcal{E}_R^{k,h,*}$ can intersect.
\end{proof}

Finally, we turn to the proof of \ref{theorem:ZPoints:UniformDist}, which is the content of the next lemma.
 
\begin{lemma}
  For every dyadic rectangle $R\subset [0,1)^2$ with $|R| \geq 2^{-2m-1}$ we have
  $$
    \#(R \cap \mathcal{Z}) \lesssim 2^{2m}mk |R|.
  $$
\end{lemma}
\begin{proof}

First, we may assume that $|R|=2^{-2m-1}$ and show
$$
 \#(R \cap \mathcal{Z}) \lesssim mk.
$$
Indeed, if it is larger we can just write it as a disjoint union of dyadic rectangles of area $2^{-2m-1}$ and use the estimate for rectangles of that size. Since $|R| = 2^{-2m-1}$ we know that it contains exactly one point $p_0\in\mathcal{P}$, and from Lemma \ref{lemma:Zcard} we know $R$ contains about $k$ points $z \in \mathcal{Z}$ such that $p(z)=p_0$ 

For every $z \in R \cap \mathcal{Z}$ such that $p(z)\not=p_0$, there exists $T
  \in \mathcal{E}$ for which $z_T = z$, which we denote by $T_z$. Since both
  $R$ and $T_z$ are dyadic rectangles and $|T_z| < |R|$ we must have 
$$
  \mathcal{Z} \cap R = \mathcal{V} \sqcup \mathcal{H} \sqcup \mathcal{O},
$$
where $\mathcal{V}$ consists of those points $z \in \mathcal{Z} \cap R$ such
  that $T_z$ intersects $R$ vertically, $\mathcal{H}$ consists of those
  points $z \in \mathcal{Z} \cap R$ such that $T_z$ intersects $R$
  horizontally, and $\mathcal{O}$ is the collection of points $z \in
  \mathcal{Z}$ whose associated point $p(z) = p_0$. We will only estimate
  the size of $\mathcal{V}$, since the argument for
  $\mathcal{H}$ is similar.

By the construction of $\mathcal{Z}$ and the previous arguments, for each $z \in \mathcal{V}$ there exists $T\in \mathcal{E}_R^{\ell,v,*}$ containing $z$. By Lemma \ref{lemma:Zcard} it suffices to show that there are no more than some constant times $m$ such rectangles.

Let $\tilde{R}=\tilde{I} \times \tilde{J}$ be the unique rectangle in $\mathcal{E}$ satisfying 
$$ 
p_R \in \tilde{R} \subset R \mbox{ and } \pi_1(\tilde{R}) = \pi_1(R).
$$
Observe that any rectangle $T \in \mathcal{E}_R^{\ell,v,*}$ intersecting $R$
  vertically and whose $z_T$ is inside $R$ must intersect $(\tilde{I})_{(1)} \times \tilde{J}$ vertically (since othwerwise $z_T$
  would be too close to $p_0$).
  According to Lemma \ref{lemma:findMaximalRectangles} with $\ell=1$, the
  maximal rectangle with shortest height must have height at least $4
  |\tilde{J}|=|\pi_2(R)|$ and the heights of these maximal rectangles increase
  exponentially. Therefore the number of maximal rectangles going through
  $(\tilde{I})_{(1)}$ is at most
$$
\log_2(2^{2m}|\pi_2(R)|) \leq 2m + \log_2(|\pi_2(R)|) \leq 2m.
$$

\end{proof}

The proof of Theorem \ref{theorem:ZPoints} is complete.


\section{Extensions and open questions} \label{section:FurtherRemarks}

\subsection{Failure of sparse bound for bi-parameter martingale transform}
In this subsection, we extend our main theorem to the bi-parameter martingale
transform, which is a $0$-complexity dyadic shift that resembles the behavior
of the bi-parameter Hilbert transform. In general, given a sequence
$\sigma=\{\sigma_R\}_{R\in\mathcal{D}\times \mathcal{D}}$ satisfying
$|\sigma_R|\leq 1$, the corresponding martingale transform is defined as
\[
T_\sigma(f):=\sum_{R\in\mathcal{D}}\sigma_R\langle f,h_R\rangle h_R.
\]We have the following lower bound result.
\begin{theorem}\label{thm:mt}
Let measures $\mu,\nu$ be the same as above. Then there exists a martingale transform $T_\sigma$ such that
\[
|\langle T_\sigma(\mu), \nu\rangle|\gtrsim 2^k.
\]
\end{theorem}
Recalling the
upper estimate of sparse forms obtained in Subsection
\ref{section:UpperBounds}, one immediately gets the following corollary:
\begin{corollary}
There cannot be any $(1,1)$ sparse domination for all bi-parameter martingale transforms.
\end{corollary} 

\begin{proof}[Proof of Theorem \ref{thm:mt}]
We first assume that $\sigma_R=0$ unless $R\subset [0,1]^2$. In general, for any $R=I\times J\in \mathcal{E}$ that gives rise to some point $z\in\mathcal{Z}$, denote $T_R:=\widehat{I_{(\ell)}}\times \widehat{J_{(k)}}$, which is the smallest rectangle that contains both $z$ and $p(z) \in \mathcal{P}$. Also, denote the standard rectangle that gives rise to $z\in\mathcal{Z}$ by $R_z$. Now for any fixed $z\in \mathcal{Z}$, 
\[
T_\sigma(\mu)(z)=\frac{1}{\#\mathcal{P}}\sum_{p\in \mathcal{P}}\sum_{R\subset [0,1]^2}\sigma_R h_R(p) h_R(z).
\]
Set $\sigma_{T}=0$ unless $T=T_R$ for some $R\in \mathcal{E}$. In particular, $\{T_R\}_{R\in\mathcal{E}}$ and $\{R\}_{R\in\mathcal{E}}$ are one-to-one, and each $T_R$ contains exactly one $p_R\in\mathcal{P}$. Hence,
\[
T_\sigma(\mu)(z)\sim 2^{-2m}\sum_{R\in\mathcal{E}} \sigma_{T_R}h_{T_R}(p_R) h_{T_R}(z).
\]

Note that $z$ cannot be contained in $T_R$ unless $R=R_z$ according to Lemma \ref{lem:TRz} below, therefore, 
\[
T_\sigma(\mu)(z)\sim 2^{-2m} \sigma_{T_{R_z}} h_{T_{R_z}}(p_{R_z})h_{T_{R_z}}(z)=2^{-2m}\sigma_{T_{R_z}}|T_{R_z}|^{-1}\xi_{T_{R_z}},
\]where $\xi_{T_{R_z}}=\pm 1$. Choosing $\sigma_{T_{R_z}}=\xi_{T_{R_z}}$, one has
\[
T_{\sigma}(\mu)(z)\sim 2^{-2m}|T_{R_z}|^{-1}=2^{-2m}\cdot \left(2^{-k+1}|R_z|\right)^{-1}\sim 2^k,
\]and the desired estimate follows immediately.
\end{proof}

In the proof above, we have used the key observation that even though $R\in \mathcal{E}$ can contain many points $z\in\mathcal{Z}$ with $p_{R}=p_{R_z}$, $T_R$ can only contain one $z_R\in \mathcal{Z}$. This is justified by the following result.
\begin{lemma}\label{lem:TRz}
For any integers $k,\ell \geq 1$ and $R=I\times J\in\mathcal{E}$, let $z_R\in \mathcal{Z}$ be the point chosen in $R$ as in Subsection \ref{subsection:buildingZ}, and $T_R=\widehat{I_{(\ell)}}\times \widehat{J_{(k)}}$. Then
\[
z_T\notin T_R,\quad \forall R, T\in \mathcal{E}, \, R\neq T.
\]
\end{lemma}
\begin{proof}
Fix $R=I\times J\in \mathcal{E}$ and denote $z=z_R$. Given any other $R'=I'\times J'\in \mathcal{E}$, our goal is to show $z_{R'}\notin T_R$. Obviously, if $p_{R'}\neq p_{R}$, then by Lemma \ref{lemma:pointFinder2} $z_{R'}$ is not even contained in $R$. It thus suffices to assume $p_{R'}=p_{R}=p$. Suppose $|I|>|I'|$, i.e. $I=I'^{(j)}$ for some $j\geq 1$, then $|J|=2^{-j}|J'|$ and $J\subset J'$. 
Since both $\widehat{J_{(k)}}, \widehat{J'_{(k)}}$ contain $\pi_2(p)$, one has $\widehat{J'_{(k)}}\supsetneq \widehat{J_{(k)}}$. 

We claim that $\pi_2(z_{R'})\notin \widehat{J_{(k)}}$, hence $z_{R'}\notin T_R$. Indeed, $\pi_2(z_{R'})\in J'_{(k)}$, and $\widehat{J_{(k)}}$ is either contained in $J'_{(k)}$ or its dyadic sibling. But it is impossible for $\widehat{J_{(k)}}$ to be contained in $J'_{(k)}$ as it would imply $\pi_2(p)\in J'_{(k)}$, which is an obvious contradiction. Therefore, $\pi_2(z_{R'})\notin \widehat{J_{(k)}}$ and the proof for the case $|I|>|I'|$ is complete.

The case $|I|<|I'|$ can be treated symmetrically, as one has $|J|<|J'|$ and can show in the same way as above that $\pi_1(z_{R'})\notin \widehat{I_{(\ell)}}$.
\end{proof}

\begin{remark}
Using the same method, one can easily show that there exists a bi-parameter dyadic shift of any given complexity which cannot have a $(1,1)$ sparse bound. We omit the details.
\end{remark}


\subsection{Failure of sparse bound in higher dimensions}
For $n\geq 1$, define the $n$-parameter strong maximal function
\[
\mathcal{M}_n(f)(x):=\sup_{R\ni x}\frac{1}{|R|}\int_R |f(y)|\,dy,\quad f:\mathbb{R}^n\to \mathbb{C},
\]where $R$ is any $n$-dimensional dyadic rectangle. Clearly, $\mathcal{M}_S$ equals $\mathcal{M}_2$. We have the following theorem.
\begin{theorem}
Let $(D_n)$ denote the following statement: there exists $C>0$ such that for all compactly supported integrable functions $f,g$ on $\mathbb{R}^n$, there exists a sparse collection of dyadic rectangles $\mathcal{S}$ such that
\[
\left|\langle \mathcal{M}_n(f), g\rangle\right|\leq C\sum_{R\in\mathcal{S}}\langle |f|\rangle_R\langle |f|\rangle_R |R|.
\]Then for all $n\geq 2$, there holds
\[
(D_n)\implies (D_{n-1}).
\]
\end{theorem}

\begin{proof}
The desired result follows from a tensor product argument. For the sake of simplicity, we will only prove $(D_3)\implies (D_2)$. Our goal is to show $(D_2)$: for any given functions $f,g$ on $\mathbb{R}^2$, we would like to find a sparse dominating form for $\langle \mathcal{M}_2(f),g\rangle$. Define $\widetilde{f}=f\otimes \chi_{[0,1)}$ and $\widetilde{g}=g\otimes \chi_{[0,1)}$, both are compactly supported integrable functions on $\mathbb{R}^3$. By the assumption $(D_3)$, there exists a sparse collection $\widetilde{\mathcal{S}}$ of rectangles in $\mathbb{R}^3$ such that
\[
|\langle M_3(\widetilde{f}), \widetilde{g}\rangle|\leq C\sum_{\widetilde{R}=R\times J\in\widetilde{\mathcal{S}}} \langle |\widetilde{f}|\rangle_{\widetilde{R}}\langle |\widetilde{g}|\rangle_{\widetilde{R}} |\widetilde{R}|,
\]where $R\in \mathbb{R}^2, J\in \mathbb{R}$.

By definition, for all $\widetilde{x}=(x,x_3)\in \mathbb{R}^3$ with $x_3\in [0,1)$,
\[
M_3(\widetilde{f})(\widetilde{x})=\sup_{R\times J\ni \widetilde{x}} \frac{1}{|R|\cdot|J|}\int_{R\times J}|\widetilde{f}(\widetilde{y})|\,d\widetilde{y}=M_2(f)(x) M(\mathbbm{1}_{[0,1)})(x_3)=M_2(f)(x).
\]Therefore,
\[
\langle M_3(\widetilde{f}), \widetilde{g}\rangle_{\mathbb{R}^3}=\langle M_2(f), g\rangle_{\mathbb{R}^2}.
\]It thus suffices to show that the sparse form in $\mathbb{R}^3$ can be dominated by a sparse form in $\mathbb{R}^2$.

To see this, rewrite
\[
\sum_{\widetilde{R}=R\times J\in\widetilde{\mathcal{S}}} \langle |\widetilde{f}|\rangle_{\widetilde{R}}\langle |\widetilde{g}|\rangle_{\widetilde{R}} |\widetilde{R}|=\sum_{\widetilde{R}=R\times J\subset \mathbb{R}^3} \alpha_{\widetilde{R}} \langle |\widetilde{f}|\rangle_{\widetilde{R}}\langle |\widetilde{g}|\rangle_{\widetilde{R}} |\widetilde{R}|
\]where $\alpha_{\widetilde{R}}=1$ if $\widetilde{R}\in\widetilde{\mathcal{R}}$, and $0$ otherwise. The sparsity of $\widetilde{\mathcal{R}}$ thus means for some constant $\Lambda$ there holds
\[
\sum_{\widetilde{R}\subset \widetilde{\Omega}}\alpha_{\widetilde{R}}|\widetilde{R}|\leq \Lambda |\widetilde{\Omega}|,\quad \forall \text{ open set }\widetilde{\Omega}\subset\mathbb{R}^3.
\]One can further simplify the expression
\[
\begin{split}
&\sum_{\widetilde{R}=R\times J\subset \mathbb{R}^3} \alpha_{\widetilde{R}} \langle |\widetilde{f}|\rangle_{\widetilde{R}}\langle |\widetilde{g}|\rangle_{\widetilde{R}} |\widetilde{R}|\\
=&\sum_{R\subset\mathbb{R}^2}\langle |f|\rangle_R \langle |g|\rangle_R |R| \sum_{J\subset \mathbb{R}} \alpha_{R\times J} |J| \cdot \left(\frac{|[0,1)\cap J|}{|J|}\right)^2\\
=:& \sum_{R\subset\mathbb{R}^2}\langle |f|\rangle_R \langle |g|\rangle_R |R| \beta_R.
\end{split}
\]It suffices to show that $\{\beta_R\}_{R\subset\mathbb{R}^2}$ is a Carleson sequence.

Indeed, decomposing
\[
\beta_R=\sum_{j=0}^{\infty} 2^{-2j}\sum_{\substack{J\subset \mathbb{R}\\ |J\cap [0,1)|\sim 2^{-j}|J|}} \alpha_{R\times J} |J|,
\]one has for all open set $\Omega\subset\mathbb{R}^2$ that
\[
\begin{split}
\sum_{R\subset \Omega} \beta_R|R| &=\sum_{j=0}^{\infty} 2^{-2j} \sum_{R\subset\Omega} \sum_{\substack{J\subset \mathbb{R}\\ |J\cap [0,1)|\sim 2^{-j}|J|}} \alpha_{R\times J} |R\times J|\\
&\leq\sum_{j=0}^{\infty} 2^{-2j} \sum_{R\times J\subset \widetilde{\Omega}} \alpha_{R\times J} |R\times J|,
\end{split}
\]where $\widetilde{\Omega}:=\Omega\times \{x_3\in \mathbb{R}:\, M(\chi_{[0,1)})(x_3)> C2^{-j}\}$. By the weak $L^1$ boundedness of $M$, one has $|\widetilde{\Omega}|\lesssim 2^j|\Omega|$, hence
\[
\sum_{R\subset \Omega} \beta_R|R| \lesssim \Lambda\sum_{j=0}^{\infty} 2^{-j} |\Omega|\lesssim \Lambda |\Omega|.
\]The proof is complete.
\end{proof}

Similar results hold true for multi-parameter martingale transforms as well. One can define $\widetilde{f}=f\otimes h_{[0,1)}$ and $\widetilde{g}=g\otimes h_{[0,1)}$. Then one observes that $|\langle T_\sigma(h_{[0,1)}), h_{[0,1)}\rangle |\geq C$ for some martingale transform $T_\sigma$ in the third variable and therefore can argue similarly as above. The details are left to the reader.


\subsection{Other sparse forms}
As mentioned in the introduction, it is also interesting to consider the possibility of domination of the bi-parameter operators by bigger sparse forms, for instance those of type $(\phi,1)$:
\[
\sum_{R\in\mathcal{S}}\langle |f|\rangle_{R,\phi}\langle |g|\rangle_R |R|,
\]where
\[
  \langle f \rangle_{R,\phi} = \inf\Biggl\{ \lambda > 0:\, \frac{1}{|R|}
    \int_R \phi\Biggl( \frac{f(x)}{\lambda} \Biggr) \leq 1 \Biggr\} \, dx,
\]and interesting norms include for example $\phi(x)=x[\log(e+x)]^\alpha$,
$\alpha>0$, and $\phi(x)=x^p$, $p>1$. We can actually extend our main theorem to the following:

\begin{theorem}\label{thm:sec5norm}
Let $\phi(x)=x[\log(e+x)]^\alpha$, $0<\alpha<\frac{1}{2}$. Then there cannot be any $(\phi,1)$-sparse domination for the strong maximal function or the martingale transform $T_\sigma$.
\end{theorem}

\begin{proof}
Fix $k\ll m$. We need to first slightly modify the measure $\mu$ to get a function that lies in the correct normed space. Let $\nu$ be the same as before and define
\[
f=\frac{1}{\#\mathcal{P}}\sum_{p\in\mathcal{P}}\frac{\mathbbm{1}_{Q_p}}{|Q_p|},
\]
where for each $p\in\mathcal{P}$, $Q_p\ni p$ is the little cube (of side-length $\sim2^{-2m}$) that is constructed in the proof of Theorem \ref{theorem:PPoints}. By the exact same argument in Section 2, it is easy to check that the same lower bound for the bilinear forms associated to $\mathcal{M}_S$ and the special martingale transform $T_\sigma$ still hold true, i.e.
\[
\langle M_S(f),\nu\rangle\gtrsim 2^k,\quad \langle T_\sigma(f), \nu\rangle\gtrsim 2^k.
\]
We shall now focus on showing the upper bound. For every $\Lambda$-Carleson collection $\mathcal{S}$ we will show that
\begin{equation}\label{eqn:sec5uppbd}
\sum_{R\in\mathcal{S}}\langle f\rangle_{R,\phi}\langle \nu\rangle_{R}|R|\lesssim \Lambda k m^\alpha\Bigl(1+\frac{2^{2k}}{m}\Bigr).
\end{equation}
If $(\ref{eqn:sec5uppbd})$ holds true, then by taking $m=2^{2k}$ and by noting that $\alpha<\frac{1}{2}$, one immediately sees that $\Lambda$ blows up to infinity as $k\to \infty$, which completes the proof.

We now prove (\ref{eqn:sec5uppbd}). Compared to the estimate of $\langle \mu\rangle_R$ in the proof of Theorem \ref{theorem:UpperBound}, it suffices to show that there is at most an extra factor of $m^\alpha$ in the upper estimate of the new average form $\langle f\rangle_{R,\phi}$. Specifically, in the case $|\overline{R}|=|R\cap [0,1)^2|\geq 2^{-2m-1}$, we would like to show that
\begin{equation}\label{eqn:sec5large}
\langle f\rangle_{R,\phi}\lesssim \frac{|\overline{R}|}{|R|}\cdot \max\left(m^\alpha, \left[\log\big(\frac{|R|}{|\overline{R}|}\big)\right]^\alpha\right).
\end{equation}Note that if this estimate holds true, then one can proceed as in the proof of Theorem \ref{theorem:UpperBound} to conclude that contribution to $\sum_{R\in\mathcal{S}}\langle f\rangle_{R,\phi}\langle \nu\rangle_R|R|$ from rectangles of this type is controlled by $\Lambda k (1+m^\alpha)$. Similarly, in the case $|\overline{R}|<2^{-2m-1}$, we will show that
\begin{equation}\label{eqn:sec5small}
\langle f\rangle_{\overline{R},\phi}\lesssim\frac{m^\alpha}{2^{2m}|\overline{R}|},
\end{equation}which, combined with the estimates for large rectangles, completes the proof of (\ref{eqn:sec5uppbd}).

To see why (\ref{eqn:sec5large}) holds, let $\lambda>0$ and one has
\[
\begin{split}
\frac{1}{|R|}\int_R \phi\left(\frac{f(x)}{\lambda}\right)\,dx&=\frac{1}{|R|}\sum_{p\in\mathcal{P}\cap \overline{R}}\int_{R\cap Q_p}\phi\left(\frac{\sum_{q\in\mathcal{P}}\frac{\mathbbm{1}_{Q_q}(x)}{|Q_q|}}{\lambda\cdot \#\mathcal{P}}\right)\,dx\\
&=\frac{1}{|R|}\cdot \#\{\mathcal{P}\cap \overline{R}\}\cdot |R\cap Q_p| \cdot\phi\left(\frac{1}{\lambda\cdot\#\mathcal{P}\cdot |Q_p|}\right)\\
&\leq \frac{|\overline{R}|}{|R|}\cdot 2^{-2m}\cdot \phi\left(\frac{2^{2m}}{\lambda}\right).
\end{split}
\]
Take $\lambda=\frac{|\overline{R}|}{|R|}\cdot \max\left(m^\alpha, \left[\log\big(\frac{|R|}{|\overline{R}|}\big)\right]^\alpha\right)$, it suffices to show that the expression above with this $\lambda$ value is $\lesssim 1$. Indeed, note that $e\ll \frac{2^{2m}}{\lambda}$, so
\[
\frac{|\overline{R}|}{|R|}\cdot 2^{-2m}\cdot \phi\left(\frac{2^{2m}}{\lambda}\right) \lesssim \frac{\left[2m+\log\big(\frac{|R|}{|\overline{R}|}\big)-\alpha\log\left(\max\big(m, \log(|R|/|\overline{R}|)\big)\right)\right]^\alpha}{\max\left(m^\alpha, \left[\log\big(\frac{|R|}{|\overline{R}|}\big)\right]^\alpha\right)} \lesssim 1.
\]

To see (\ref{eqn:sec5small}), again let $\lambda>0$ and one can calculate similarly as above to obtain
\[
\frac{1}{|\overline{R}|}\int_{\overline{R}} \phi\left(\frac{f(x)}{\lambda}\right)\, dx\leq 2^{-2m}\phi\left(\frac{2^{2m}}{\lambda}\right)=\frac{1}{\lambda}\left[\log\big(e+\frac{2^{2m}}{\lambda}\big)\right]^\alpha.
\]
Take $\lambda=\frac{m^\alpha}{2^{2m}|\overline{R}|}$, then it suffices to show that the expression above with this $\lambda$ value is $\lesssim 1$. Indeed, the lower bound $|\overline{R}|\gtrsim 2^{-2m-k}$ guarantees that $\frac{2^{2m}}{\lambda}\gg e$, hence
\[
\frac{1}{\lambda}\left[\log\big(e+\frac{2^{2m}}{\lambda}\big)\right]^\alpha\lesssim \frac{2^{2m}|\overline{R}|}{m^\alpha}\left[4m+\log|\overline{R}|-\alpha\log m\right]^\alpha\lesssim 2^{2m}|\overline{R}|\lesssim 1.
\]
\end{proof}

The argument in the theorem above is not strong enough to produce a contradiction when $\alpha\geq \frac{1}{2}$. In other words, it could be possible that the bi-parameter operators that are considered have a $(\phi,1)$-type sparse bound with some $\alpha\geq \frac12$. Moreover, it is easy to check that a very similar argument as in Theorem \ref{thm:sec5norm} barely fails for concluding a contradiction when $\phi(x)=x^p$, $p>1$. We think that positive or negative results for those types of bounds would be very interesting.

\begin{problem*}
Determine whether $(\phi,1)$-type sparse bounds for $\mathcal{M}_S$ are possible for $\phi(x)=x^p$, $p>1$ or $\phi(x)=x[\log(e+x)]^\alpha$ ($\alpha\geq \frac12$). 
\end{problem*}

\bibliography{bibliography}
\bibliographystyle{abbrv}

\end{document}